\documentclass[12pt, reqno]{amsart}
\usepackage[margin = 1.3in]{geometry}   
\usepackage[english]{babel}
\usepackage{amssymb,amsmath,amsthm}
\usepackage{array}
\usepackage[shortlabels]{enumitem}
\RequirePackage{tikz}[2010/10/13] 
\usetikzlibrary{cd}

\usepackage{xspace}\xspaceaddexceptions{\,}

\makeatletter
\def\@listI{\leftmargin\leftmargini
    \parsep\parskip \itemsep -\parsep
    \itemsep 2pt}
\let\@listi\@listI
\makeatother

\theoremstyle{definition}
\newtheorem{definition}{Definition}[section]
\newtheorem{example}[definition]{Example}

\newtheorem{remark}[definition]{Remark}

\theoremstyle{plain}
\newtheorem{lemma}[definition]{Lemma}
\newtheorem{proposition}[definition]{Proposition}
\newtheorem{theorem}[definition]{Theorem}
\newtheorem{corollary}[definition]{Corollary}

\newcommand\A{{\mathbf A}}
\newcommand\B{{\mathbf B}}

\newcommand\G{{\mathbf G}}

\newcommand\I{{\mathbf I}}
\newcommand\M{{\mathbf M}}

\newcommand\3{{\mathbf 3}}

\renewcommand\S{{\mathbf S}}
\newcommand\two{\mathbf B_2}
\newcommand\WK{{\mathbf{WK}}}
\newcommand\HomGR{\mathrm{Hom_{_{GR}}}(\A,\3)}
\newcommand\Homb{\mathrm{Hom_{_{b}}}(\S,\3)}

\newcommand\IBSL{\ensuremath{\mathcal{IBSL}}\xspace}

\newcommand\PWK{\ensuremath{\mathrm{PWK}}\xspace}

\DeclareMathAlphabet{\mathbfsf}{\encodingdefault}{\sfdefault}{bx}n
\makeatletter
\providecommand*{\Dashv}{\mathrel{\mathpalette\@Dashv\vDash}}
\newcommand*{\@Dashv}[2]{\reflectbox{$\m@th#1#2$}}
\makeatother

\renewcommand\geq{\geqslant}

\newcommand\leqs{\leqslant}

\usepackage{units}

\newcommand\pair[1]{{\langle#1\rangle}}

\newcommand\PL{{\mathcal{P}_{l}}}
\newcommand\PLA{{\mathcal{P}_{l} (\mathbb{A})}}
\newcommand\PLB{{\mathcal{P}_{l} (\mathbb{B})}}

\newcommand\cIBSL{{\mathfrak{IBSL}}}

\newcommand\SA{{\mathfrak{SA}}}

\newcommand\phih{{\varphi_{_h}}}
\useshorthands{"}
\defineshorthand{"-}{{}\nobreakdash-\hspace{0pt}}	

\usepackage[authormarkup=none]{changes}

\definechangesauthor[color=red]{Stef}

\usepackage[authormarkup=none]{changes}

\definechangesauthor[color=blue]{comm}

\begin{document}

\title[A duality for involutive bisemilattices] {A duality for involutive bisemilattices}

\author[S.Bonzio]{Stefano Bonzio}

\author[A.Loi]{Andrea Loi}

 \author[L.Peruzzi]{Luisa Peruzzi}


\subjclass[2000]{Primary 08C20. Secondary 06E15, 18A99, 22A30.}
\keywords{duality, bisemilattice, Stone space, Boolean algebra, P\l onka sum.}
\date{\today}

\begin{abstract}
We establish a duality between the category of involutive bisemilattices and the category of semilattice inverse systems of Stone spaces, using Stone duality from one side and the representation of involutive bisemilattices as P\l onka sum of Boolean algebras, from the other.
Furthermore, we show that the dual space of an involutive bisemilattice can be axiomatized as a GR space with involution, 
a generalization of the spaces introduced by Gierz and Romanowska,
equipped with an involution as additional operation.
\end{abstract}

\maketitle

\section{Introduction}

It is a common trend in mathematics to study dualities for general algebraic structures and, in particular, for those arising from mathematical logic. The first step towards this direction traces back to the pioneering work by Stone for Boolean algebras \cite{Stone37}. Later on, Stone duality has been extended to the more general case of distributive lattices by Priestley \cite{Priestley72}. The two above mentioned are the prototypical examples of dualities obtained via \emph{dualizing} objects and will be both recalled and constructively used in the present work. 

These kind of dualities have an intrinsic value: they are indeed a way of describing the very same mathematical object from two different perspectives, the target category and its dual. More generally, dualities between algebraic structures and corresponding topological spaces may open the way to applications as algebraic problems can possibly be translated into topological ones, or new insights can be obtained via the representation of a particular algebra as an algebra of continuos functions over a certain space (for a more detailed exposition of applications see \cite{Davey1993,clark1998natural}).

The starting point of our analysis is the duality established by Gierz and Romanowska \cite{Romanowska} between distributive bisemilattices and compact totally disconnected partially ordered left normal bands with constants, which, for sake of compactness, we refer to as GR spaces. The duality is obtained via the usual strategy of finding a suitable candidate to play the role of dualizing, or schizofrenic, object. However, the relevance of the result lies mainly in the use, for the first time, of P\l onka sums as an essential tool for stating the duality.

Our aim is to provide a duality between the categories of involutive bisemilattices and those topological spaces, here christened as GR spaces with involution. The former consists of a class of algebras introduced and extensively studied in \cite{Bonzio16} as algebraic semantics (although not equivalent) for  paraconsistent weak Kleene logic. The logical interests around these structures is relatively recent; on the other hand, it is easily checked that involutive bisemilattices, as introduced in \cite{Bonzio16}, are equivalent to the regularization of the variety of Boolean algebras, axiomatized by P\l onka \cite{Plonka69,Plo84}.\footnote{The authors were not aware of some of the mentioned results by P\l onka when writing \cite{Bonzio16}.} For this reason, involutive bisemilattices are strictly connected to Boolean algebras as they are representable as P\l onka sums of Boolean algebras.
 
The present work consists of two main results. On the one hand, taking advantage of the P\l onka sums representation in terms of Boolean algebras and Stone duality, we are able to describe the dual space of an involutive bisemilattice as semilattice inverse systems of Stone spaces (Theorem \ref{th: dualita ibsl}). On the other hand, we generalize Gierz and Romanowska duality by considering GR spaces with involution as an additional operation (Theorem \ref{th: stron-inv SA equivalente a GR inv}).  As a byproduct of our analysis we get a topological description of \emph{semilattice inverse systems} of Stone spaces (Corollary \ref{corollario principale}).\\
The paper is structured as follows. In Section \ref{sec: preliminari} we summarize all the necessary notions and known results about bisemilattices, Gierz and Romanowska duality and involutive bisemilattices. In Section \ref{sec: dir,inv systems} we the categories of \emph{semilattice} direct and inverse systems, proving that, when constructed using dually equivalent categories, they are also dually equivalent. In Section \ref{sec: dualita}, we introduce GR spaces with involution and prove the main results.  Finally, in Section \ref{sec: commenti} we make some considerations about categories admitting both topological duals and a representation in terms of P\l onka sums. By using Priestley duality, we then extend our results to the category of distributive bisemilattices.

\section{Preliminaries}\label{sec: preliminari}

A \emph{distributive bisemilattice} is an algebra $\mathbf{A}=\langle A, +, \cdot\rangle $ of type $\langle 2,2\rangle$ such that both $ + $ and $\cdot$ are idempotent, associative and commutative operations and, moreover, $+$ ($\cdot$ respectively) distributes over $\cdot$ (+ respectively). Distributive bisemilattices, originally called  ``quasi-lattices'', have been introduced by P\l onka \cite{Plo67a}; nowadays, similar structures are studied in a more general setting under the name of Birkhoff systems (see \cite{Harding2017}, \cite{Harding20172}). Throughout the paper we will refer to these algebras simply as \emph{bisemilattices}. Observe that every distributive lattice is an example of bisemilattice and every semilattice is a bisemilattice, where the two operations coincide. Any bisemilattice induces two different partial orders, namely $x\leq_{_\cdot} y$ iff $x\cdot y = x$ and $x\leq_{_+} y$ iff $x+y = y$.
\begin{example}\label{ex: algebra 3}
The 3-element algebra $\mathbf{3}=\langle \{0,1,\alpha\},\cdot,+\rangle $, whose operations are defined by the so-called \emph{weak Kleene tables} (given below), is the most prominent example of bisemilattice, as it generates the variety of (distributive) bisemilattices \cite{Kal71}.

\vspace{10pt}
\begin{center}\renewcommand{\arraystretch}{1.25}
\begin{tabular}{>{$}c<{$}|>{$}c<{$}>{$}c<{$}>{$}c<{$}}
   \cdot & 0 & \alpha & 1 \\[.2ex]
 \hline
       0 & 0 & \alpha & 0 \\
       \alpha & \alpha & \alpha & \alpha \\          
       1 & 0 & \alpha & 1
\end{tabular}
\qquad
\begin{tabular}{>{$}c<{$}|>{$}c<{$}>{$}c<{$}>{$}c<{$}}
   + & 0 & \alpha & 1 \\[.2ex]
 \hline
     0 & 0 & \alpha & 1 \\
     \alpha & \alpha & \alpha & \alpha\\          
     1 & 1 & \alpha & 1
\end{tabular}
\end{center}

\vspace{10pt}

The two partial orders induced by $\3$ are displayed in the following Hasse diagrams:
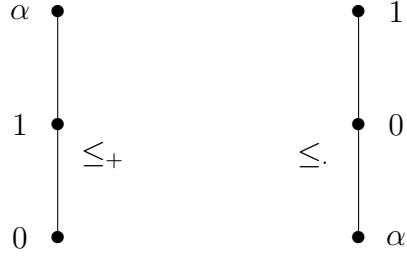
\begin{figure}[h]
\begin{tikzpicture}\label{fig:1}

\draw (0,0) -- (0,1.5); 
\draw (0,1.5) -- (0,3); 
\draw (0,0) node {$\bullet$};
\draw (-0.5, 0) node {0};
\draw (0,1.5) node {$ \bullet$};
\draw (-0.5, 1.5) node {1};
\draw (0.6, 1.1) node {$\leq_{+}$};
\draw (0,3) node {$\bullet$};
\draw (-0.5, 3) node {$\alpha$}; 

\draw (4,0) -- (4,1.5); 
\draw (4,1.5) -- (4,3); 
\draw (4,0) node {$\bullet$};
\draw (4.5,0) node {$\alpha$};
\draw (4,1.5) node {$ \bullet$};
\draw (4.5,1.5) node {0};
\draw (3.4, 1.1) node {$\leq_{\cdot}$};
\draw (4,3) node {$\bullet$};
\draw (4.5, 3) node {1};

\end{tikzpicture}
\end{figure}
\end{example} 

\vspace{20pt}


A duality for bisemilattices has been established in \cite{Romanowska}, by using $\3$ as \emph{dualizing} object. We recall here all the notions needed to state the main result.\\
A \emph{left normal band} is an idempotent semigroup $\langle A, \ast \rangle$ satisfying the additional identity $x\ast (y\ast z)\approx x\ast (z\ast y)$, which is weak form of commutativity. A left normal band can be equipped with a partial order. 
\begin{definition}\label{def: p.o.l.n.b.}
A \emph{partially ordered left normal band} is an algebra $\mathbf{A}=\langle A, \ast, \leq \rangle $ such that 
\begin{itemize}
\item[i)] $\langle A, \ast \rangle $ is a left normal band
\item[ii)] $\langle A, \leq\rangle $ is a partially ordered set
\item[iii)] if $ x\leq y $ then $x\ast z\leq y\ast z$ and $z\ast x\leq z\ast y$
\item[iv)] $x\ast y \leq x$
\end{itemize}
\end{definition}

In any partially ordered left normal band it is possible to define a second partial order via $\ast $ and $\leq$: $ a\sqsubseteq b $ iff $a\ast b\leq b$ and $b\ast a = b$. A partially ordered left normal band may be also extended by adding constants. 
\begin{definition}\label{def: left normal band with constants}
A \emph{partially ordered left normal band with constants} is an algebra $\mathbf{A}=\langle A, \ast , \leq , c_0 , c_1 , c_{\alpha}\rangle$ such that $\langle A, \ast , \leq \rangle$ is a partially ordered left normal band and $c_0$, $c_1$ and $c_\alpha$ are constants satisfying
\begin{itemize}
\item[(1)] $x\ast c_{\alpha} = c_{\alpha}\ast x = c_{\alpha}$
\item[(2)] $x\ast c_0 = x\ast c_1 = x$
\item[(3)] $c_0\sqsubseteq x\leq c_1 $ and $c_{\alpha} \leq x \sqsubseteq c_{\alpha}$
\item[(4)] if $c_0\ast x = c_1 \ast x$ then $x=c_{\alpha}$
\end{itemize}
\end{definition}

\begin{definition}\label{def: GR-space}
A \emph{GR space} is a structure $\mathbf{A} =\langle A, \ast, \leq, c_0, c_1, c_\alpha , \tau \rangle $, such that $\langle A, \ast , \leq , c_0 , c_1 , c_\alpha \rangle $ is a partially ordered left normal band with constants and $\tau$ is a topology making $ \ast:A\times A \rightarrow A $ a continous map and $ \langle A, \leq , \tau \rangle $ is a compact  \emph{totally order disconnected} space\footnote{A topological space is totally order disconnected if (1) $\{ (a,b)\in A\times A: a\leq b \}$ is closed; (2) if $a\not\leq b$ then there is an open and closed lower set $ U $ such that $ b\in U $ and $ a\not\in U $.}. 
\end{definition}

\begin{example}
The support set of $\mathbf{3} $, namely $\{0, 1,\alpha\}$ equipped with the discrete topology, where $ \leq\;\equiv \;\leq_{\cdot}$, $ c_0 = 0 $, $c_1 = 1$, $c_\alpha = \alpha $ and $ \ast $ is defined as follows: $$a\ast b=\begin{cases} a & \mbox{if }b\neq \alpha \\ b & \mbox{otherwise }
\end{cases} $$
is a GR space (it is not difficult to check that operation $a\ast b = a + a\cdot b = a\cdot (a+b)$ and that the induced order $\sqsubseteq$ coincides with $\leq_{_+}$). 
\end{example}
We call $\mathfrak{DB}$ the category of bisemilattices (whose morphisms are homomorphisms of bisemilattices) and $\mathfrak{GR}$ the category of GR spaces (whose morphisms are continuous maps preserving $\ast$, constants and the order).  
The above mentioned duality is stated as follows:
\begin{theorem}\cite[Theorem 7.5]{Romanowska}\label{th: dualita Romanowska}
The categories $\mathfrak{DB}$ and $\mathfrak{GR}$ are dual to each other under the invertible functor $\mathrm{Hom_{b}}(-,\3):\mathfrak{DB}\to\mathfrak{GR}$ and its inverse $\mathrm{Hom_{_{GR}}}(-,\3):\mathfrak{GR}\to\mathfrak{DB}$.
\end{theorem}
  
In detail, given a bisemilattice $\mathbf{S}$, its dual GR space is $ \widehat{\mathbf{S}} = \mathrm{Hom_{b}}(\mathbf{S},\mathbf{3}) $, i.e. the space of the homomorphisms (of bisemilattices) from $ \mathbf{S} $ to $\mathbf{3}$. Analogously, if $ \mathbf{A} $ is a GR space, then the dual is given by $ \widehat{\mathbf{A}}=\mathrm{Hom_{_{GR}}}(\mathbf{A},\mathbf{3}) $, the bisemilattice of morphisms of $\mathfrak{GR}$. \\

The isomorphism between $\S$ and $\widehat{\widehat{\S}}$ is given by: 
\begin{equation}\label{eq: isomorfismo S, Scapcap}
\varepsilon_{_S}: \S\to\widehat{\widehat{\S}}, x\mapsto \varepsilon_{_S}(x), \varepsilon_{_S}(x)(\varphi)=\varphi (x),  
\end{equation}
for every $x\in S$ and $\varphi\in \widehat{S}$. 

Analogously, for $\A$ and $\widehat{\widehat{\A}}$, the isomorphism is given by:
\begin{equation}\label{eq: isomorfismo A, Acapcap}
\delta_{_A}: \A\to\widehat{\widehat{\A}}, x\mapsto \delta_{_A}(x), \delta_{_A}(x)(\varphi)=\varphi (x), 
\end{equation} 
for every $x\in A$ and $\varphi\in \widehat{A}$. 

The class of \emph{involutive bisemilattices} has been introduced in \cite{Bonzio16} as the most suitable candidate to be the algebraic counterpart of PWK logic. 

\begin{definition}\label{def: IBSL}
\normalfont
An \emph{involutive bisemilattice} is an algebra $\B = \pair{B,\cdot,+,^{'},0,1}$ of type $(2,2,1,0,0)$ satisfying:
\vspace{5pt}
\begin{enumerate}[label=\textbf{I\arabic*}.]
\item $x + x\approx x$;
\item $x + y\approx y+ x$;
\item $x+(y+ z)\approx (x+ y)+ z$;
\item $(x')'\approx x$;
\item $x\cdot y\approx ( x'+ y')'$;
\item $x\cdot( x'+ y)\approx x\cdot y$; \label{rmp}
\item $0+ x\approx x$;
\item $1\approx 0'$.
\end{enumerate}
\end{definition}
We denote the variety of involutive bisemilattices by $\mathcal{IBSL}$. 

Every involutive bisemilattice has, in particular, the structure of a join semilattice with zero, in virtue of axioms~(I1)--(I3) and~(I7). More than that, it is possible to prove \cite[Proposition 20]{Bonzio16} that $ \cdot $ distributes over $ + $ and viceversa, therefore the reduct $\pair{B,+,\cdot}$ is a bisemilattice. 
Notice that, in virtue of axioms~(I5) and~(I8), the operations $\cdot$ and $1$ are completely determined by $+$, $^{'}$, and $0$. 
It is not difficult to check that every involutive bisemilattice has also the structure of a meet semilattice with $1$, and that the equations $x+ y \approx ( x'\cdot y')' $, 
$x+ y \approx x+( x'\cdot y) $ 
are satisfied. There are different equivalent ways to define involutive bisemilattices: it is not difficult to check that $\IBSL$ corresponds to the regularization of the variety of Boolean algebras described in \cite{Pl'onka84}.

\begin{example}
\normalfont
Every Boolean algebra, in particular the 2-element Boolean algebra $\two$, is an involutive bisemilattice. Also, the 2-element semilattice with zero, which we call $\S_2$, endowed with identity as its unary fundamental operation, is an involutive bisemilattice. The most prominent example of involutive bisemilattice is the 3-element algebra $\WK$, which is obtained by expanding the language of $\mathbf{3}$ with an involution behaving as follows:
\vspace{10pt}
\begin{center}\renewcommand{\arraystretch}{1.25}
\begin{tabular}{>{$}c<{$}|>{$}c<{$}}
  ^{'} &  \\[.2ex]
\hline
  1 & 0 \\
  \alpha & \alpha \\
  0 & 1 \\
\end{tabular}
\end{center}
\vspace{15pt}

 Upon considering the partial order $\leq_{\cdot}$ induced by the product in its bisemilattice reduct, it becomes a 3"-element chain with $ \alpha$ as its bottom element. $ \two $, $ \S_2 $ and $\WK $ can be represented by means of the following Hasse diagrams (the dashes represent the order, while the arrows represent the negation):
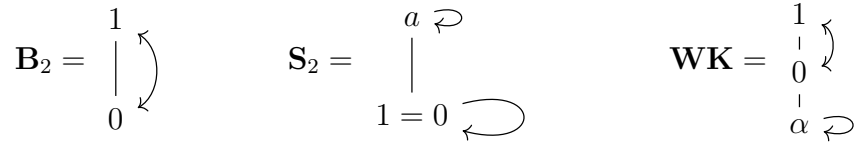
\[
\two = \begin{tikzcd}
  1 \arrow[d, bend left = 60, leftrightarrow] \\ 0 \arrow[u, dash]
  \end{tikzcd}
  \qquad\qquad
\S_2 =  \begin{tikzcd}
 a \arrow[loop right] \\ 1 = 0\arrow[u, dash]\arrow[loop right]
  \end{tikzcd}
  \qquad\qquad
\WK = \begin{tikzcd}[row sep = tiny]
 1 \arrow[d, bend left = 70, leftrightarrow] \\ 0\arrow[u, dash]  \\ \alpha \arrow[loop right] \arrow[u, dash] 
  \end{tikzcd}
\]
\vspace{10pt}
It is not difficult to verify that $ \two $ is a subalgebra of $ \WK $, while $ \S_2 $ is a quotient.
\end{example}
Although the algebra $ \WK $ allows to define the logic $ \PWK $ (upon setting $\{ 1, \alpha \}$ as designated values), its relevance is a consequence also of the fact that it generates the variety $ \mathcal{IBSL} $, \cite[Corollary 31]{Bonzio16}. This result can be also proved, observing that involutive bisemilattices coincide with the regularization of Boolean algebras, axiomatized in \cite{Plo84}, and, due to \cite{Lakser72}, the only subdirectly irreducible members of the class are $ \two $, $ \S_2 $ and $\WK $.

As the main focus of this paper is introducing a duality for involutive bisemilattices, it is useful to recall here the definition of dual categories. We assume the reader is familiar with the concepts of category and morphism (in a category).
\begin{definition}\label{def: equivalenza di categorie}
Two categories $\mathfrak{C}$ and $\mathfrak{D}$ are \emph{equivalent} provided there exist two covariant functors, $\mathcal{F}:\mathfrak{C}\to\mathfrak{D}$ and $\mathcal{G}:\mathfrak{D}\to\mathfrak{C}$ such that $ \mathcal{G}\circ\mathcal{F} = id_{\mathfrak{C}} $ and $ \mathcal{F}\circ\mathcal{G} = id_{\mathfrak{D}} $.
\end{definition}

Whenever the functors considered in the above definition are \emph{controvariant} (instead of covariant), the two categories $\mathfrak{C}$ and $\mathfrak{D}$ are said to be \emph{dually equivalent} or, briefly, \emph{duals}.

\section{The categories of semilattice inverse and direct systems}\label{sec: dir,inv systems}

In this section we are going to describe a very general procedure to construct dualities for algebraic structures admitting a P\l onka sum representation.

For our purposes, we need to strengthen the well known concepts of inverse and direct system of a category, hence we introduce the notions of semilattice inverse and semilattice direct systems in a very direct and way. For sake of simplicity, we opt for presenting these topics following the current trend in algebraic topology (see \cite{shapebook} for details).

\begin{definition}\label{def: Sem inverse system}
Let $ \mathfrak{C} $ be an arbitrary category, a \emph{semilattice inverse system} in the category $ \mathfrak{C} $ is a tern $\mathcal{X}= \pair{X_{i}, p_{ii'}, I} $ such that
\vspace{3pt} 
\begin{itemize}
\item[(i)] $ I $ is a join semilattice with lower bound;
\item[(ii)] for each $ i\in I $, $ X_i $ is an object in $ \mathfrak{C}$;
\item[(iii)] $p_{ii'} : X_{i'} \to X_i$ is a morphism of $ \mathfrak{C} $, for each pair $i\leqs i'$, satisfying that $p_{ii} $ is the identity in $ X_i $ and such that $ i\leq i'\leq i'' $ implies $ p_{ii'} \circ p_{i'i''}  = p_{ii''}$. 
\end{itemize}
\end{definition}

$ I $ is called the \emph{index set} of the system $ \mathcal{X} $, $ X_i $ are the \emph{terms} and $ p_{ii'} $ are referred to as \emph{bonding morphisms} of $ \mathcal{X} $.  For convention, we indicate with $ \vee $ the semilattice operation on $ I $, $ \leq $ the induced order and $ i_0$ the lower bound in $ I $.

The only difference making an inverse system a semilattice inverse system is the requirement on the index set to be a semilattice with lower bound instead of a directed preorder. 

\begin{definition}\label{def: morfismi di Sem inv}
Given two semilattice inverse systems $ \mathcal{X}=\pair{X_{i}, p_{ii'}, I} $ and $ \mathcal{Y}=\pair{Y_{j}, q_{jj'}, J} $, a \emph{morphism} between $ \mathcal{X} $ and $\mathcal{Y} $ is a pair $  (\varphi, f_j) $ such that
\vspace{3pt}
\begin{itemize}
\item[i)]  $ \varphi :J\rightarrow I $ is a semilattice homomorphism;
\item[ii)] for each $ j\in J $, $ f_j: X_{\varphi(j)}\rightarrow Y_{j} $ is a morphism in $ \mathfrak{C} $, such that whenever $ j\leq j' $, then the diagram in Fig.3 commutes. 
\begin{center}
\begin{figure}[h]\label{fig: diagramma semilattice inv systems}
\begin{tikzpicture}
\draw (-4,0) node {$ Y_{j} $};
	\draw (4,0) node {$ Y_{j'} $};
	
	\draw [line width=0.8pt, <-] (-3.6,0) -- (3.6,0);
	\draw (0,-0.4) node {\begin{footnotesize}$q_{jj'}$\end{footnotesize}};
	
	\draw [line width=0.8pt, ->] (3.3,3) -- (-3.3,3);
	\draw (0, 3.3) node {\begin{footnotesize}$p_{\varphi(j)\varphi(j')}$\end{footnotesize}};

	\draw (-4,3) node {$X_{\varphi(j)}$}; 
	
	\draw (4,3) node {$X_{\varphi(j')}$};
		
	\draw [line width=0.8pt, ->] (-4,2.6) -- (-4,0.4);
	\draw (-4.6,1.7) node {\begin{footnotesize}$f_j$\end{footnotesize}};
	\draw (4.6,1.7) node {\begin{footnotesize}$f_{j'}$\end{footnotesize}};
	\draw [line width=0.8pt, <-] (4,0.4) -- (4,2.6);
	
\end{tikzpicture}
\caption{The commuting diagram defining morphisms of semilattice inverse systems.}
\end{figure}
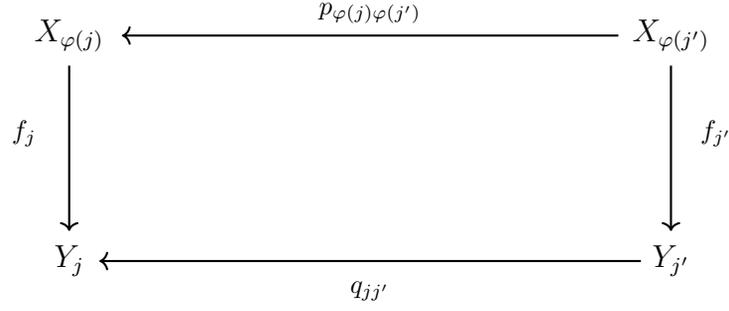
\end{center}
\end{itemize}
\end{definition}
Notice that, for morphisms of semilattice inverse systems, 
the assumption that $ \varphi:J\rightarrow I $ is a (semilattice) homomorphism implies that whenever $ j\leq j' $ then $ \varphi(j)\leq \varphi(j') $.
Given three semilattice inverse systems $ \mathcal{X}=\pair{X_{i}, p_{ii'}, I} $, $ \mathcal{Y}=\pair{Y_{j}, q_{jj'}, J} $, $ \mathcal{Z} = \pair{Z_{k}, r_{kk'}, K} $, the composition of morphisms is defined in the same way as for inverse systems.
\begin{lemma}\label{lem: composizione di morfismi}
The composition of morphisms between semilattice inverse systems is a morphism.
\end{lemma}
\begin{proof}
Let $ (\varphi, f_j): \mathcal{X}\rightarrow\mathcal{Y} $, $ (\psi, g_k):\mathcal{Y}\rightarrow\mathcal{Z} $, then $(\chi, h_k)=(\psi, g_k)(\varphi, f_j):\mathcal{X}\rightarrow\mathcal{Z} $ is $ \chi = \varphi\psi $, $ h_{k}=g_{k}f_{\chi(k)} $. $ \chi $ is the composition of two (semilattice) homomorphisms, hence it is a semilattice homomorphism. The claim follows from the commutativity of the following diagram (we omitted the indexes for the maps $ p,q, r, f,g $ to make the notation less cumbersome)
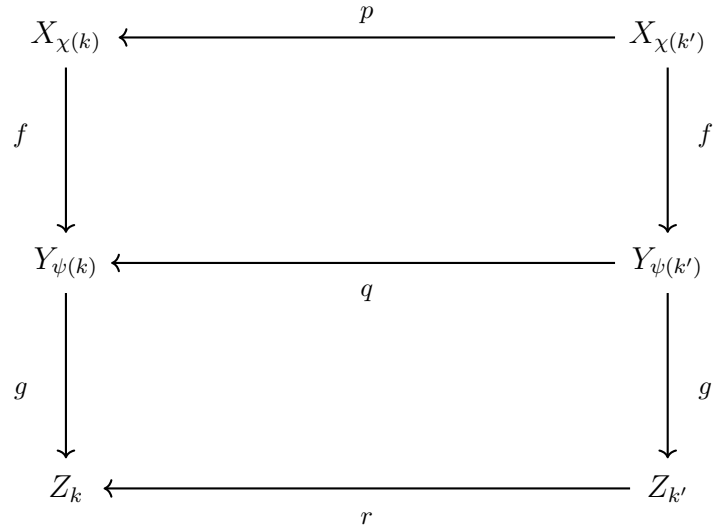
\begin{center}
\begin{tikzpicture}
\draw (-4,-3) node {$ Z_{k} $};
	\draw (4,-3) node {$ Z_{k'} $};
\draw [line width=0.8pt, <-] (-3.5,-3) -- (3.5,-3);
	\draw (0,-3.4) node {\begin{footnotesize}$r$\end{footnotesize}};
	
	\draw [line width=0.8pt, <-] (-4,-2.6) -- (-4,-0.4);
	\draw (-4.6,-1.7) node {\begin{footnotesize}$g$\end{footnotesize}};
	\draw (4.5,-1.7) node {\begin{footnotesize}$g$\end{footnotesize}};
	\draw [line width=0.8pt, ->] (4,-0.4) -- (4,-2.6);

\draw (-4,0) node {$ Y_{\psi(k)} $};
	\draw (4,0) node {$ Y_{\psi(k')} $};
	
	\draw [line width=0.8pt, <-] (-3.4,0) -- (3.3,0);
	\draw (0,-0.4) node {\begin{footnotesize}$q$\end{footnotesize}};
	
	\draw [line width=0.8pt, ->] (3.3,3) -- (-3.3,3);
	\draw (0, 3.3) node {\begin{footnotesize}$p$\end{footnotesize}};

	\draw (-4,3) node {$X_{\chi(k)}$}; 
	
	\draw (4,3) node {$X_{\chi(k')}$};
		
	\draw [line width=0.8pt, ->] (-4,2.6) -- (-4,0.4);
	\draw (-4.6,1.7) node {\begin{footnotesize}$f$\end{footnotesize}};
	\draw (4.5,1.7) node {\begin{footnotesize}$f$\end{footnotesize}};
	\draw [line width=0.8pt, <-] (4,0.4) -- (4,2.6);
	
\end{tikzpicture}
\end{center}
\end{proof}

\begin{proposition}\label{prop: categoria Sem inv}
Let $\mathfrak{C} $ an arbitrary category. Then \emph{Sem-inv-}$\mathfrak{C}$ is the category whose objects are semilattice inverse systems in $ \mathfrak{C} $ with morphisms as defined above. 
\end{proposition}
\begin{proof}
The composition of morphisms between systems is associative and the identity morphism is $ (1_{I},1_{i}) $, where $ 1_{I}: I\rightarrow I $ is the identity homomorphism on $ I $ and $ 1_i: X_{i}\rightarrow X_i $ is the identity morphism in the category $ \mathfrak{C} $. 
\end{proof}

The category of \emph{semilattice direct systems} of a given category $ \mathfrak{C} $ is obtained by reversing morphisms of Sem-inv-$\mathfrak{C}$ as follows:
\begin{definition}\label{def: direct system}
 Let $ \mathfrak{C} $ be an arbitrary category. A \emph{semilattice direct system} in $ \mathfrak{C} $ is a triple $\mathbb{X}= \pair{X_{i}, p_{ii'}, I} $ such that 
 \begin{itemize}
\item[(i)] $ I $ is a join semilattice with least element.
\item[(ii)] $ X_i $ is an object in $ \mathfrak{C}$, for each $ i\in I $;
\item[(iii)] $p_{ii'} : X_{i} \to X_{i'}$ is a morphism of $ \mathfrak{C} $, for each pair $i\leqs i'$, satisfying that $p_{ii} $ is the identity in $ X_i $ and such that $ i\leq i'\leq i'' $ implies $ p_{i'i''} \circ p_{ii'}  = p_{ii''}$. 
\end{itemize}
\end{definition}

We call $ I $, $X_i$, the index set and the terms of the direct system, respectively, while we refer to $ p_{ii'} $ as \emph{transition morphisms} to stress the crucial difference with respect to inverse systems.  

A morphism between two semilattice direct systems $ \mathbb{X} $ and $ \mathbb{Y} $ is a pair $(\varphi, f_i):\mathbb{X}\to\mathbb{Y}$ s. t.
\begin{itemize}
\item[i)] $ \varphi: I\rightarrow J $ is a semilattice homomorphism 
\item[ii)] $ f_i: X_{i}\rightarrow Y_{\varphi(i)} $ is a morphism of $ \mathfrak{C} $, making the following diagram commutative for each $ i, i'\in I $, $ i\leq i' $:
\end{itemize}
\begin{center}
\begin{figure}[h]\label{fig: diagramma semilattice dir systems}
\begin{tikzpicture}
\draw (-4,0) node {$ Y_{\varphi(i)} $};
	\draw (4,0) node {$ Y_{\varphi(i')} $};
	
	\draw [line width=0.8pt, ->] (-3.5,0) -- (3.4,0);
	\draw (0,-0.4) node {\begin{footnotesize}$q_{\varphi(i)\varphi(i')}$\end{footnotesize}};
	
	\draw [line width=0.8pt, <-] (3.3,3) -- (-3.3,3);
	\draw (0, 3.3) node {\begin{footnotesize}$p_{ii'}$\end{footnotesize}};

	\draw (-4,3) node {$X_i$}; 
	
	\draw (4,3) node {$X_{i'}$};
		
	\draw [line width=0.8pt, ->] (-4,2.6) -- (-4,0.4);
	\draw (-4.6,1.7) node {\begin{footnotesize}$f_i$\end{footnotesize}};
	\draw (4.6,1.7) node {\begin{footnotesize}$f_{i'}$\end{footnotesize}};
	\draw [line width=0.8pt, <-] (4,0.4) -- (4,2.6);
	
\end{tikzpicture}
\caption{The commuting diagram defining morphisms of semilattice direct systems.}
\end{figure}
\end{center}
The composition of two morphisms is defined as $(f_{i},\varphi)(g_{j},\psi)= (h_{i},\chi)$, 
$$ \chi = \psi\varphi , \;\;\;\;\;\;\;\;\;\;\; h_{i}=g_{\varphi(i)}f_{i}:X_{i}\rightarrow Z_{\chi(i)}. $$
It is easily verified that the composition $(h_{i},\chi)$ is a morphism and it is associative and that the element $ (1_{I},1_{i}) $, where  $ 1_{I}: I\rightarrow I $ is the identity map on $ I $ and $ 1_i: X_{i}\rightarrow X_i $ is the identity morphism in $ \mathfrak{C} $, is the identity morphism between semilattice direct systems. Therefore semilattice direct systems form a category which we will call Sem-dir-$\mathcal{C}$.

In the remaining part of this section we aim to show that the categories of semilattice direct and semilattice inverse systems of dually equivalent categories are dually equivalent. 
In order to do that, given a controvariant functor $\mathcal{F}\colon\mathfrak{C}\to\mathfrak{D}$ between two categories $\mathfrak{C}$ and $\mathfrak{D}$, we define a new functor $\mathcal{\widetilde{F}}\colon \text{ Sem-dir-}\mathfrak{C}\to\text{ Sem-inv-}\mathfrak{D}$ as follows 

\begin{equation}\label{eq: deinifione F tilde, G tilde}
\mathcal{\widetilde{F}}(\mathbb{X}):= \pair{\mathcal{F}(X_i), \mathcal{F}(p_{ii'}), I} \;\;\;\;\;\;\;\;\;\;\mathcal{\widetilde{F}}(\varphi,f_i):= (\varphi, \mathcal{F}(f_{i})), 
\end{equation}
\vspace{3pt}

where $\mathbb{X}= \pair{X_i, p_{ii'}, I} $ is an object and $(\varphi,f_i)$ a morphism in the category Sem-dir-$\mathfrak{C}$.

Similarly, whenever $\mathcal{G}\colon\mathfrak{D}\to\mathfrak{C}$ is a controvariant functor, then we define $\mathcal{\widetilde{G}}\colon \text{ Sem-inv-}\mathfrak{D}\to\text{ Sem-dir-}\mathfrak{C}$ as in \eqref{eq: deinifione F tilde, G tilde}. The crucial point is proving that the new maps are indeed functors.

\begin{lemma}\label{lemma: funtori Sem-dir Sem-inv}
Let $\mathcal{F}:\mathfrak{C}\to\mathfrak{D}$ be a controvariant functor between two categories $\mathfrak{C}$ and $\mathfrak{D}$. Then: 
\begin{enumerate}
\item $\widetilde{\mathcal{F}}$ is a controvariant functor between Sem-dir-$\mathfrak{C} $ and  Sem-inv-$\mathfrak{D} $; 
\item $\widetilde{\mathcal{G}}$ is a controvariant functor between Sem-inv-$\mathfrak{C} $ and  Sem-dir-$\mathfrak{D} $.
\end{enumerate}
\end{lemma}
\begin{proof}
Proof of (1) and (2) are essentially analogous, so we just give the details of (1): the reader can check that they can easily adapted to prove (2). \\
Assume that $\mathbb{X}=\pair{X_i, p_{ii'}, I}$ is an object in Sem-dir-$\mathfrak{C}$. We first show that $\widetilde{\mathcal{F}}(\mathbb{X})$ is an object in Sem-inv-$\mathfrak{D}$, namely it satisfies conditions (i), (ii), (iii) of Definition \ref{def: Sem inverse system}. Recall that, by \eqref{eq: deinifione F tilde, G tilde}, $ \mathcal{\widetilde{F}}(\mathbb{X}):= \pair{\mathcal{F}(X_i), \mathcal{F}(p_{ii'}), I} $. 
\begin{itemize}
\item[(i)] is clearly satisfied as $I$ is a semilattice with lower bound; 
\item[(ii)] Since $X_i$ is an object in $\mathfrak{C}$ and $\mathcal{F}$ a functor, $\mathcal{F}(X_i)$ is an object of $\mathfrak{D}$;
\item[(iii)] Let $i\leq i'$. Then there exists a morphism $p_{ii'} : X_{i} \to X_{i'}$ of $ \mathfrak{C} $ such that  $p_{ii} $ is the identity in $ X_i $ and moreover, if $ i\leq i'\leq i'' $ then $ p_{i'i''} \circ p_{ii'}  = p_{ii''}$. Since $\mathcal{F}$ is a controvariant functor, $\mathcal{F}(p_{ii'}): \mathcal{F}(X_{i'})\to\mathcal{F}(X_i)$ is a morphism of $\mathfrak{D}$. Moreover, $\mathcal{F}(p_{ii}) =\mathcal{F}(1_{\mathfrak{C}}) = 1_{\mathfrak{D}} $ and compositions are obviously preserved. 

\end{itemize}
This shows that $\widetilde{\mathcal{F}}(\mathbb{X})$ is an object. Now, suppose that $(\varphi, f_i) $ is a morphism in Sem-dir-$\mathfrak{C}$ between two objects $\mathbb{X}=\pair{ X_i, p_{ii'}, I}$, $\mathbb{Y}=\pair{ Y_j, p_{jj'}, J}$. We show that $\mathcal{\widetilde{F}}(\varphi,f_i):= (\varphi, \mathcal{F}(f_{i})) $ is a morphism from $\mathcal{\widetilde{F}}(\mathbb{Y})=\pair{\mathcal{F}(Y_j),\mathcal{F}(q_{jj'}), J} $ to $\mathcal{\widetilde{F}}(\mathbb{X})= \pair{\mathcal{F}(X_i),\mathcal{F}(p_{ii'}), I}$ (which also assures that $\widetilde{\mathcal{F}}$ is controvariant). We check that $\mathcal{\widetilde{F}}(\varphi,f_i)$ satisfies the properties i) and ii) in Definition \ref{def: morfismi di Sem inv}.
\begin{itemize}
\item[i)] clearly holds as $\varphi:I\to J$ is a semilattice homorphism from the index set of $\mathcal{\widetilde{F}}(\mathbb{X})$ to the index set of $\mathcal{\widetilde{F}}(\mathbb{Y})$; 
\item[ii)] For every $i\in I$, $f_i: X_{i}\to Y_{\varphi(i)}$ is a morphism in $\mathfrak{C}$ making the Diagram in Fig.1 commutative. Therefore, $\mathcal{F}(f_i): \mathcal{F}(Y_{\varphi(i)})\to\mathcal{F}(X_i)$ is a morphism in $\mathfrak{D}$. Suppose that $i\leq i'$, for some $i,i'\in I$, then $p_{ii'}: X_i\to X_{i'} $ and $\mathcal{F}(p_{ii'}):\mathcal{F}(X_{i'})\to\mathcal{F}(X_{i})$ is the correspondent morphism in $\mathfrak{D}$. Since $\varphi$ is a semilattice homomorphism, we have $\varphi(i)\leq\varphi(i')$, and $\mathcal{\widetilde{F}}(\mathbb{X})$ is a semilattice inverse system, then the following diagram commutes:
\begin{center}
\begin{tikzpicture}
\draw (-4,0) node {$ \mathcal{F}(X_{i}) $};
	\draw (4,0) node {$ \mathcal{F}(X_{i'}) $};
	
	\draw [line width=0.8pt, <-] (-3.1,0) -- (3.1,0);
	\draw (0,-0.4) node {\begin{footnotesize}$\mathcal{F}(p_{ii'})$\end{footnotesize}};
	
	\draw [line width=0.8pt, ->] (3,3) -- (-3,3);
	\draw (0, 3.3) node {\begin{footnotesize}$\mathcal{F}(q_{\varphi(i)\varphi(i')})$\end{footnotesize}};

	\draw (-4,3) node {$\mathcal{F}(Y_{\varphi(i)})$}; 
	
	\draw (4,3) node {$\mathcal{F}(Y_{\varphi(i')})$};
		
	\draw [line width=0.8pt, ->] (-4,2.6) -- (-4,0.4);
	\draw (-4.6,1.7) node {\begin{footnotesize}$\mathcal{F}(f_i)$\end{footnotesize}};
	\draw (4.6,1.7) node {\begin{footnotesize}$\mathcal{F}(f_{i'})$\end{footnotesize}};
	\draw [line width=0.8pt, <-] (4,0.4) -- (4,2.6);
	
\end{tikzpicture}
\end{center}
\end{itemize}
This concludes our claim.
\end{proof}

Notice that the statement of Lemma \ref{lemma: funtori Sem-dir Sem-inv} is false when considering \emph{covariant} functors instead of controvariant as shown by the following example.

\begin{example}
Let $\mathfrak{C}$ be an algebraic category, $\mathfrak{Set}$ the category of sets and $\mathcal{F}\colon\mathfrak{C}\to\mathfrak{Set}$ the forgetful functor. For any object $\mathbb{X}=\pair{X_{i}, p_{ii'}, I}$ in Sem-dir-$\mathfrak{C}$, $\mathcal{\widetilde{F}}(\mathbb{X})$ is not an object in Sem-inv-$\mathfrak{Set}$. Indeed, for any two indexes such that $i\leq i'$, we have a morphism in $\mathfrak{C}$, $p_{ii'}\colon X_{i}\to X_{i'}$; since $\mathcal{F}$ is covariant, $\mathcal{\widetilde{F}}(p_{ii'})=\mathcal{F}(p_{ii'})$ is a function (a morphism in $\mathfrak{Set}$) from $\mathcal{F}(X_i)$ to $\mathcal{F}(X_{i'})$, hence it does not fulfill condition (iii) in Definition \ref{def: Sem inverse system}.
\end{example}

\begin{theorem}\label{teorema: duality Sem-dir Sem-inv}
Let $\mathfrak{C}$ and $\mathfrak{D}$ be dually equivalent categories. Then \emph{Sem-dir-}$\mathfrak{C}$ and \emph{Sem-inv-}$\mathfrak{D}$ are dually equivalent.
\end{theorem}
\begin{proof}
By hypothesis we have two controvariant functors $\mathcal{F}$ and $\mathcal{G}$ 

\[
\begin{tikzcd}[row sep = tiny, arrows = {dash}]
& & \mathcal{F} &  & & \\
  & \mathfrak{C}\arrow[rr, bend left = 25, rightarrow] &  & \mathfrak{D}\arrow[ll, bend left = 25, rightarrow] &                                  &  \\
& & \mathcal{G} & & &   
  \end{tikzcd}
  \]
such that such that $ \mathcal{G}\circ\mathcal{F} = id_{\mathfrak{C}} $ and $ \mathcal{F}\circ\mathcal{G} = id_{\mathfrak{D}} $. By Lemma \ref{lemma: funtori Sem-dir Sem-inv} we have controvariant functors $\mathcal{\widetilde{F}}$ and $\mathcal{\widetilde{G}}$ 
\[
\begin{tikzcd}[row sep = tiny, arrows = {dash}]
& & \mathcal{\widetilde{F}} &  & & \\
& & & & & \\
& & & & & \\
  & \text{Sem-dir-}\mathfrak{C}\arrow[rr, bend left = 25, rightarrow] &  & \text{Sem-inv-}\mathfrak{D}\arrow[ll, bend left = 25, rightarrow] &                                  &  \\
  & & & & & \\
& & & & & \\
& & \mathcal{\widetilde{G}} & & &   
  \end{tikzcd}
  \] 
  
We only need to check that the compositions $ \mathcal{\widetilde{G}}\circ\mathcal{\widetilde{F}} $ and $ \mathcal{\widetilde{F}}\circ\mathcal{\widetilde{G}} $ are the identities in the categories Sem-dir-$\mathfrak{C}$ and Sem-inv-$\mathfrak{D}$, respectively. Let $\mathbb{X}=\pair{X_i, p_{ii'}, I}$ be an object in Sem-dir-$\mathfrak{C}$. Then 
\[
\mathcal{\widetilde{G}}(\mathcal{\widetilde{F}}(\mathbb{X}))= \mathcal{\widetilde{G}}(\pair{\mathcal{F}(X_i), \mathcal{F}(p_{ii'}), I}) =\pair{\mathcal{G}\circ\mathcal{F} (X_i), \mathcal{G}\circ\mathcal{F}(p_{ii'}), I } = \pair{X_i, p_{ii'}, I},
\]
where the last equality is obtained by $ \mathcal{G}\circ\mathcal{F} = id_{\mathfrak{C}} $. It is analogous to verify that $ \mathcal{\widetilde{F}}\circ\mathcal{\widetilde{G}} $ is the identity. 
\end{proof}

The above result somehow resembles \emph{semilattice-based dualities} establish by Romanowska and Smith in \cite{Romanowska96,Romanowska97}, where the authors essentially show how to lift a duality between two categories, in particular an algebraic category and its dual representation spaces, to a duality involving the correspondent semilattice representations. The duality in Theorem \ref{teorema: duality Sem-dir Sem-inv} is also ``based'' on certain semilattice systems. However, the two approaches are characterized by a substantial difference: Romanowska and Smith indeed consider, from one side, the semilattice sum of an algebraic category but, on the other, the semilattice representation of the dual spaces, and thus the duality, is obtained by \emph{dualizing} the semilattice of the index sets. In order to achieve this, they rely on the duality due to Hofmann, Mislove and Stralka \cite{hofmann1974pontryagin} for semilattices (see also \cite{Davey1993} for details). This means, that the semilatttice representation of the dual spaces (of the considered categories) is constructed via compact topological semilattices with $0$ which carries the Boolean topology (namely makes the space compact, Hausdorff and totally disconnected).

\section{The category of Involutive Bisemilattices and its dual}\label{sec: dualita}

P\l onka introduced \cite{Plo67,Plo68, Plo84} a construction to build algebras out of semilattice systems of algebras\footnote{It is essential to have a semilattice with lower bound (instead of a pointless semilattice) when working with algebras having constants, see \cite{Plo84}.}, see also \cite{RomanPlonka,romanowska2002modes}

\begin{definition}\label{def: somma di Plonka}
Let $ \mathbb{A}=\pair{\A_i , \varphi_{ii'}, I} $ be a semilattice direct system of algebras $\A_i=\pair{A_i , f_{t}}$ of a fixed type $\nu$, then the \emph{P\l onka sum} over $\mathbb{A}$ is the algebra $\PLA = \pair{\bigsqcup_I A_i, g^{\mathcal{P}}_{_t}} $, whose universe is the disjoint union and the operations $ g^{\mathcal{P}}_{_t} $ are defined as follows: for every $n$"-ary $g_{_t}\in \nu$, and $a_1,\dots, a_n\in \bigsqcup_I A_i $, where $n\geq 1$ and $a_r\in A_{i_r}$, we set $j = i_1 \lor\dots\lor i_n$ and define
\[
g^{\mathcal{P}}_{_t} (a_1,\dots,a_n) = g^{\A_j}_{_t} (\varphi_{i_1j}(a_1),\dots,\varphi_{i_nj}(a_n)).
\]
In case $\nu$ contains constants, then, for every constant $g\in\nu$, we define $g^{\mathcal{P}} = g^{\A_{i_{0}}}$.
\end{definition}

Involutive bisemilattices, as well as bisemilattices admits a representation as P\l onka sums over a semilattice sistem of Boolean algebras (this was already proved in \cite{Plonka69, Plo84}). 
\begin{theorem}[\mbox{\cite[Thm.~46]{Bonzio16}}]\label{the: rappresentazione Plonka}\

\begin{itemize}
\item[1)] If $\mathbb{A} $ 
is a semilattice direct system of Boolean algebras, then the $ \PLA $ is an involutive bisemilattice.

\item[2)] If $\B$ is an involutive bisemilattice, then $\B$ is isomorphic to the P\l onka sum over a semilattice direct system of Boolean algebras\footnote{The form of the semilattice direct system used in the P\l onka sum representation is not needed for the purposes of this work. For more details, the reader could refer to \cite{Bonzio16} or \cite{Pl'onka84}.}. 
\end{itemize}
\end{theorem}
The above result states that every involutive bisemilattice admits a unique representation as P\l onka sum of Boolean algebras. 
We summarize here the categories we are dealing with
\vspace{10pt}
\begin{center}
\begin{tabular}{|c|c|c|}
\hline
\textbf{Category} & \textbf{Objects} & \textbf{Morphisms} \\
\hline
$\mathfrak{BA}$ & Boolean Algebras & Homomorphisms of $\mathcal{BA}$ \\
\hline
$\cIBSL$ & Involutive bisemilattices & Homomorphisms of $\mathcal{IBSL}$ \\
\hline
Sem-dir-$\mathfrak{BA}$ & semilattice direct systems of B.A. & Morphisms of Sem-dir-$\mathfrak{BA}$ \\
\hline
$\mathfrak{SA}$ & Stone spaces & continuous maps \\
\hline
Sem-inv-$\mathfrak{SA}$ & semilattice inverse systems of Stone sp.  & Morphisms of Sem-inv-$\mathfrak{SA}$ \\
\hline
\end{tabular}
\end{center}
\vspace{15pt}
Theorem \ref{the: rappresentazione Plonka} states that the objects of the category $\mathfrak{IBSL}$ are isomorphic to the objects of the category Sem-dir-$\mathfrak{BA}$. We aim at proving more, namely that they are also equivalent as categories. 
In order to establish this, we prove the following auxiliary lemmata. 
\begin{lemma}\label{lem: omomorfismi tra ibsl}
Let $ \mathbb{A} =\pair{\A_{i}, p_{ii'}, I} $ and $ \mathbb{B}=\pair{\B_{j}, q_{jj'},J} $ be semilattice direct systems of Boolean algebras. Let $\A=\PLA$, $\B=\PLB$ and $ h: \A\to \B $ a homomorphism, then for any $ i\in I $ there exists a $j\in J $ such that
\begin{enumerate}
\item $ h(A_i)\subseteq B_{j} $
\item $ h_{|A_{i}} $ is a Boolean homomorphism from $ \A_i $ into $ \B_{j} $  
\end{enumerate}
\end{lemma}
\begin{proof}
(1) As first notice that, from the construction of P\l onka sums, we have that for any $ x\in A_i $, also $ x'\in A_i $. Consequently, for any $ h(x)\in B_j $, for a certain $j\in J$, then also $ h(x)'\in B_j $. Let $ a\in A_i $ for some $ i\in I $, then there exists a $ j\in J $ such that $ h(a)\in B_j $. Therefore $ h(0_{A_i}) = h(a\wedge a') = h(a)\wedge h(a') = h(a)\wedge h(a)' = 0_{B_j}  $, where the last equality holds since $ h(a) $ and $ h(a)'$ belong to the same Boolean algebra $ B_j $. Similarly, $ h(1_{A_i}) = h(a\vee a') = h(a)\vee h(a') = h(a)\vee h(a)' = 1_{B_j}  $. 

We now have to prove that for any $ a\in A_i $, with $ a\neq 0_{A_i} $ we have that $ h(a)\in B_j $. 
Suppose, by contradiction, that $ a\in A_i $, and $ h(a)\in B_k$, with $ j\neq k $. Then $ 0_{B_j} = h(0_{A_i}) = h(a\wedge a') = h(a)\wedge h(a') = h(a)\wedge h(a)' = 0_{B_k} $, which is impossible, as, by construction $ B_{j}\cap B_k = \emptyset $, hence, necessarily $ h(A_i)\subseteq B_j $. 

(2) follows from the fact that $ h $ preserves joins, meets and complements by definition and we already proved that $ h(0_{A_i}) = 0_{B_j} $ and $ h(1_{A_i}) = 1_{B_j} $.
\end{proof}
Theorem \ref{the: rappresentazione Plonka} together with Lemma \ref{lem: omomorfismi tra ibsl} state that $\mathcal{IBSL}-$homomorphisms are nothing but homomorphisms between the correspondent (unique) P\l onka sum representations. 
The statement of Lemma \ref{lem: omomorfismi tra ibsl} can be exposed more precisely saying that there exists a map $ \varphi:I\rightarrow J $ such that for every homomorphism $ h: \PLA\rightarrow\PLB $, $ h(A_{i})\subseteq B_{\varphi(i)} $. It is not difficult to prove that such map is actually a semilattice homomorphism.  
\begin{lemma}\label{lem: omomorfismo degli indici}
Let $ \mathbb{A} =\pair{\A_{i}, p_{ii'}, I} $ and $ \mathbb{B}=\pair{\B_{j}, q_{jj'},J} $ be semilattice direct systems of Boolean algebras. Let $\A=\PLA$, $\B=\PLB$, $ h: \A\to \B $ a homomorphism and $ \phih:I\rightarrow J $ such that $ h(A_{i})\subseteq B_{\varphi(i)} $. Then $ \phih $ is a semilattice homomorphism.
\end{lemma}
\begin{proof}
Let $ a_1\in A_i $ and $ a_{2}\in A_{i'} $, with $ i,i'\in I $; by definition of $ \PLA $, $ a_1\wedge a_2\in A_{i\vee i'} $ and $ h(a_1)\in B_{\phih (i)} $, $ h(a_2)\in B_{\phih (i')} $, then $ h(a_{1}\wedge a_{2}) = h(a_1)\wedge h(a_2)\in B_{\phih(i)\vee\phih(i')} $. But since $ h(a_{1}\wedge a_{2})\in B_{\phih(i\vee i')} $, then necessarily $ \phih(i\vee i')= \phih(i)\vee\phih(i') $, i.e. $ \phih $ is a semilattice homorphism.
\end{proof}

\begin{lemma}\label{lemma: morfismi tra somme di Plonka}
Let $ \mathbb{A} =\pair{\A_{i}, p_{ii'}, I} $ and $ \mathbb{B}=\pair{\B_{j}, q_{jj'},J} $ be semilattice direct systems of Boolean algebras and $(\varphi, f_{i})$ a morphism from $\mathbb{A}$ to $\mathbb{B}$. Then $h\colon \PLA\to\PLB$, defined as
\[ 
h(a):= f_{i}(a),
\]
where $i\in I$ is the index such that $a\in A_i$, is a homorphism of involutive bisemilattices.
\end{lemma}
\begin{proof}
The map $h$ is well defined for every $i\in I$, as by assumption $f_i$ is homomorphism of Boolean algebras. We only have to check that $h$ is compatible with all the operations of an involutive bisemilattice. To simplify the notation we set $\A=\PLA$, $\B=\PLB$. 

As regards the constants (we give details of one of them only), let $i_0$ be the lower bound in $I$ (if follows that $\varphi(i_0)$ is the lower bound in $J$), then $
h(0)= f_{i_0}(0)= 0_{\varphi(i_0)} = 0. $ Similarly, for negation, assume $a\in A_i$, for some $i\in I$: $h(\neg a) = f_{i}(\neg a) = \neg f_{i}(a) = \neg h(a)$.

As for binary operations (we consider $\wedge$ only as the case of $\vee$ is analogous), assume $a\in A_i$, $b\in A_j$ and set $k=i\vee j$:

\[
h(a\wedge b)) = h (p_{ik}(a)\wedge^{\A_k} p_{jk}(b)) = f_{k}(p_{ik}(a)\wedge^{\A_k} p_{jk}(b)) =
\]
\[
= f_{k}(p_{ik}(a))\wedge^{\B_{\varphi(k)}} f_{k}(p_{jk}(b)) = q_{\varphi(i)\varphi(k)}(f_{i} (a))\wedge^{\B_{\varphi(k)}} q_{\varphi(j)\varphi(k)}(f_{j} (b)) = 
\]
\[ 
=(f_{i}(a)\wedge f_{j}(b)) = (h(a)\wedge h(b)),
\]
where the equality in the second line is justified by the commutativity of the following diagram, which holds for $i,j\leq k$, as, by assumption, $(\varphi, f_i)$ is morphism in Sem-dir-$\mathfrak{C}$). 

\begin{center}
\begin{tikzpicture}[scale=1]
\draw (-4,0) node {$ \B_{\varphi(i)} $};
	\draw (4,0) node {$ \B_{\varphi(k)} $};
	
	\draw [line width=0.8pt, ->] (-3.4,0) -- (3.3,0);
	\draw (0,-0.4) node {\begin{footnotesize}$q_{\varphi(i)\varphi(k)}$\end{footnotesize}};
	
	\draw [line width=0.8pt, <-] (3.3,3) -- (-3.3,3);
	\draw (0, 3.3) node {\begin{footnotesize}$p_{ik}$\end{footnotesize}};

	\draw (-4,3) node {$\A_i$}; 
	
	\draw (4,3) node {$\A_{k}$};
		
	\draw [line width=0.8pt, ->] (-4,2.6) -- (-4,0.4);
	\draw (-4.6,1.7) node {\begin{footnotesize}$f_i$\end{footnotesize}};
	\draw (4.6,1.7) node {\begin{footnotesize}$f_{k}$\end{footnotesize}};
	\draw [line width=0.8pt, <-] (4,0.4) -- (4,2.6);
	
\end{tikzpicture}
\end{center}
\end{proof}


\begin{theorem}\label{th: IBSL e Sem-dir-BA sono equivalenti}
The categories $\mathfrak{IBSL}$ and Sem-dir-$\mathfrak{BA}$ are equivalent.
\end{theorem}
\begin{proof}
The equivalence is proved by defining the following functors: 
\[
\begin{tikzcd}[row sep = tiny, arrows = {dash}]
& & \mathcal{F} &  & & \\
& & & & & \\
& & & & & \\
  & \mathfrak{IBSL}\arrow[rr, bend left = 25, rightarrow] &  & \text{Sem-dir-}\mathfrak{BA}\arrow[ll, bend left = 25, rightarrow] &                                  &  \\
  & & & & & \\
& & & & & \\
& & \mathcal{G} & & &   
  \end{tikzcd}
  \] 
$\mathcal{F}$ associates to an involutive bisemilattices $\A\cong\PLA$ (due to Theorem \ref{the: rappresentazione Plonka}), the semilattice direct system of Boolean algebras $\mathbb{A}$. On the other hand, for a semilattice direct system of Boolean algebras $\mathbb{A}$, we define $\mathcal{G}(\mathbb{A}):=\PLA$. 

We firstly check that $\mathcal{F}$ and $\mathcal{G}$ are controvariant functors.

Let $\A,\B\in\IBSL$ such that $\A\cong\PLA$ and $\B\cong\PLB$, with $\mathbb{A}=\pair{\A_i, p_{ii'}, I}$ and $\mathbb{B}=\pair{\B_j, q_{jj'}, J}$ semilattice direct systems of Boolean algebras. Then, for every IBSL-homomorphism $h\colon \A\to\B$, we define $\mathcal{F}(h):= (\varphi_{h}, h_{|A_{i}} ) $, where $\varphi_{h}$ is the semilattice homomorphism defined in Lemma and $h_{|A_{i}}$ is the restriction of $h$ on the Boolean components $\A_i$ of the P\l onka sum corresponding to $\A$. Lemmas \ref{lem: omomorfismi tra ibsl} \ref{lem: omomorfismo degli indici} guarantee that $\varphi_{h}$ is a semilattice homomorphism and that $h_{|A_{i}}$ is a Boolean homomorphism in each component $\A_i $ of the P\l onka sum $\PLA$. Moreover, the following diagram is commutative for each $i\leq i'$ (notice that $i\leq i'$ implies $\varphi_{h}(i)\leq\varphi_{h}(i')$)
\begin{center}
\begin{tikzpicture}[scale=1]
\draw (-4,0) node {$ \B_{\varphi_{h}(i)} $};
	\draw (4,0) node {$ \B_{\varphi_{h}(i')} $};
	
	\draw [line width=0.8pt, ->] (-3.3,0) -- (3.3,0);
	\draw (0,-0.4) node {\begin{footnotesize}$q_{\varphi_{h}(i)\varphi_{h}(i')}$\end{footnotesize}};
	
	\draw [line width=0.8pt, <-] (3.3,3) -- (-3.3,3);
	\draw (0, 3.3) node {\begin{footnotesize}$p_{ii'}$\end{footnotesize}};

	\draw (-4,3) node {$\A_i$}; 
	
	\draw (4,3) node {$\A_{i'}$};
		
	\draw [line width=0.8pt, ->] (-4,2.6) -- (-4,0.4);
	\draw (-4.6,1.7) node {\begin{footnotesize}$h_{|A_{i}}$\end{footnotesize}};
	\draw (4.6,1.7) node {\begin{footnotesize}$h_{|A_{i'}}$\end{footnotesize}};
	\draw [line width=0.8pt, <-] (4,0.4) -- (4,2.6);
	
\end{tikzpicture}
\end{center}

Therefore $\mathcal{F}(h)$ is a morphism from $\mathbb{A}$ to $\mathbb{B}$, showing that $\mathcal{F}$ is a covariant functor.

As concern $\mathcal{G}$, by Theorem \ref{the: rappresentazione Plonka} we know that $\PLA$ is an involutive bisemilattice. Moreover, if $(\varphi, f_i)\colon \mathbb{A}\to\mathbb{B}$ is a morphism between the semilattice direct systems of Boolean algebras $\mathbb{A}$ and $\mathbb{B}$, then set $\mathcal{G}(\varphi, f_i):= h $, as defined in Lemma \ref{lemma: morfismi tra somme di Plonka}, which assures that $\mathcal{G}(\varphi, f_i)$ is a homomorphism from $\PLA$ to $\PLB$.

We are left with verifying that the composition of the two functors gives the identity in both categories. Let $\A\in\IBSL$, such that $\A\cong\PLA $then 
\[
\mathcal{G}(\mathcal{F}(\A)) = \mathcal{G}(\mathcal{F}(\PLA)) = \mathcal{G}(\mathbb{A}) = \PLA = \A;
\]  
\[
\mathcal{F}(\mathcal{G}(\mathbb{A})) = \mathcal{F}(\PLA) = \mathbb{A}.
\]    
\end{proof}

Recall that a \emph{Stone space} is topological space which is compact, Hausdorff and totally disconnected. Stone spaces can be viewed as a category, which we refer to as $\mathfrak{SA}$, with continuous maps as morphisms.

It is well known that the category of Stone spaces is the dual of the category of Boolean algebras \cite{Stone37}.
The above statement, combined with Theorem \ref{teorema: duality Sem-dir Sem-inv}, gives immediately the following
\begin{corollary}\label{cor: Sem-dir-BA duale di Sem-inv-SA}
The categories Sem-dir-$\mathfrak{BA}$ and Sem-inv-$\mathfrak{SA}$ are dually equivalent.
\end{corollary}

By Theorem \ref{th: IBSL e Sem-dir-BA sono equivalenti}, $\cIBSL$ is equivalent to the category of semilattice direct systems of Boolean algebras. Due to Corollary \ref{cor: Sem-dir-BA duale di Sem-inv-SA}, we have then a first abstract characterization of the dual category of $\cIBSL$.

\begin{theorem}\label{th: dualita ibsl}
The category Sem-inv-$\SA$ and $ \cIBSL $ are dually equivalent.
\end{theorem}

Theorem \ref{th: dualita ibsl} gives a description of the dual category of involutive bisemilattices in terms of Stone spaces, i.e. the dual category of Boolean algebras, objects coming into play due to the representation Theorem \ref{the: rappresentazione Plonka}.  

The above theorem together with Theorem \ref{the: rappresentazione Plonka} should be compared with the following theorem due to Haimo \cite{Haimo}, where \emph{direct limits} are considered instead of P\l onka sums. In the following statement, $\lim_{\longrightarrow}$,  
$\lim_{\longleftarrow}$ denote the direct and inverse limit, respectively.
\begin{theorem}[\cite{Haimo}, Th. 9]\label{th: Haimo}
Let $\{\A_i\}$ be a direct system of Boolean algebras and $\{\A_{i}^{*}\}$ the corresponding family of Stone spaces. Then $$(\lim_{\longrightarrow} \A_i )^{*} \cong\lim_{\longleftarrow}\A_{i}^{*}.$$
\end{theorem}


In Theorem \ref{th: stron-inv SA equivalente a GR inv} (see below) we will give a \emph{concrete} topological axiomatization of the dual space of an involutive bisemilattice via Gierz and Romanowska duality (see Theorem \ref{th: dualita Romanowska}). Logically motivated by the fact that involutive bisemilattices are intrinsically characterized by an involutive negation, we aim at recovering a unary operation in the dual space: we expand GR spaces to GR spaces with involution.


\begin{definition}\label{def: GR-space with involution}
A \emph{GR space with involution} is a GR space $\G$ with a continous map $\neg :G\to G$ such that for any $a \in G $:
\begin{enumerate}[label=\textbf{G\arabic*}.]
\item $\neg (\neg a)=a$
\item $\neg(a\ast b)=\neg a\ast \neg b$
\item if $a\leq b$ then $\neg b \sqsubseteq \neg a$
\item $\neg c_0 = c_1$, $\neg c_1 = c_0$ and $\neg c_{\alpha} = c_{\alpha}$
\item The space $\mathrm{Hom_{_{GR}}}(\A,\3)$ (see Section \ref{sec: preliminari}) equipped with natural involution $\neg$, i.e. $\neg\varphi (a) = (\varphi(\neg a))'$ satisfies $\varphi\cdot (\neg\varphi + \psi) = \psi \cdot \varphi$, where operations are defined pointwise;
\item there exist $\varphi_{0}, \varphi_{1}\in\HomGR $ such that $\neg\varphi_{0}=\varphi_{1}$ and $\varphi + \varphi_{0} = \varphi$, for each $\varphi\in\HomGR$. 
\end{enumerate}
\end{definition}
\vspace{5pt}
\begin{example}\label{es: WK come GR space with involution}
$\WK $ equipped with discrete topology is the canonical example of GR space with involution. 
\end{example}
\begin{definition}\label{def: categoria IGR}
$\mathfrak{IGR}$ is the category whose objects are GR spaces with involution and whose morphisms are GR-morphisms preserving involution. 
\end{definition}

Given a GR space with involution $\G$, we can consider its GR space reduct (simply its involution free reduct), call it $\A$, which can be associated to the dual distributive bisemilattice $\widehat{\A}=\HomGR$. Aiming at turning it into an involutive bisemilattice, we define an involution on $\widehat{\A}$ as follows: 
$$\neg \Phi (a) = (\Phi(\neg a))',$$
for each $\Phi\in \widehat{\A}$ and $a\in G$, where $\neg$ and $'$ are the involutions of $\G$ and $\WK$, respectively. 

Adopting the same idea, given an arbitrary involutive bisemilattice $\mathbf{I}$, we consider its bisemilattice reduct $\S = \pair{I, +, \cdot}$, which is distributive \cite[Proposition 20]{Bonzio16}, and therefore can be associated to its dual GR space, $\widehat{\S} = \Homb$ (see Section \ref{sec: preliminari}). The bisemilattice $\3$ turns into $\WK$ just by adding the usual involution and the constants 0, 1,  so it makes sense to define an involution on $\widehat{\S}$ as: $$\neg\varphi (x)=(\varphi(x'))', $$ for any $\varphi\in\widehat{\S}$ and $x\in S$.

\begin{lemma}\label{lem: chiusura di A^ rispetto all'involuzione}
Let $\mathbf{G}$ be a GR-space with involution and $\A$ its GR-space reduct. Then, if $\Phi\in\widehat{\A}$ then $\neg\Phi\in\widehat{\A}$. Moreover, $\widehat{\G}=\pair{\widehat{\A},\neg} $ is an involutive bisemilattice.
\end{lemma} 
\begin{proof}
Assuming that $\Phi$ is a morphism of GR spaces, we have to verify that also $\neg\Phi$ is, i.e. that it is a continuos map, preserving operation $\ast$, constants and the order $\leq$. Observe that $\neg\Phi$ is continuous as it is the composition of continuous maps. 

Concerning operations and constants, we have: \\
$\neg\Phi (a\ast b) = (\Phi \neg(a\ast b))' = (\Phi(\neg a\ast\neg b))'= (\Phi (\neg a)\ast\Phi (\neg b))' = (\Phi (\neg a))'\ast (\Phi (\neg b))'= \neg \Phi (a)\ast \neg \Phi (b)$. \\
 $\neg\Phi (c_0) = (\Phi (\neg c_0))' = (\Phi (c_1))'= 1' = 0$. 
 
Similarly, $\neg\Phi (c_1) = (\Phi (\neg c_1))' = (\Phi (c_0))'= 0' = 1$ and $\neg\Phi (c_\alpha) = (\Phi (\neg c_\alpha))' = (\Phi (c_\alpha))'= \alpha' = \alpha$. 
 
As for the order, let $a\leq b$, but then $\neg b \sqsubseteq \neg a$. Since $\Phi$ preserve both the orders, $\Phi (\neg b)\leq_{+}\Phi (\neg a)$, thus $(\Phi (\neg a))'\leq_{\cdot}(\Phi (\neg a))'$, i.e. $\neg\Phi ( a)\leq\neg\Phi ( b)$. 

To prove that $\widehat{G}$ is an involutive bisemilattice, we have to check that conditions \textbf{I1} to \textbf{I8} of Definition \ref{def: IBSL} hold for $\widehat{\G}$. Clearly, \textbf{I1}, \textbf{I2} and \textbf{I3} hold as $\widehat{\A}$ is a distributive bisemilattice, while \textbf{I6}, \textbf{I7} and $\textbf{I8}$ hold by definition. For the remaining ones, let $\Phi,\Psi\in\widehat{\A}$ and $a\in A$. \\
\textbf{I4.} $\neg (\neg \Phi (a)) = \neg\Phi(\neg (a))'= \Phi (\neg\neg a)'' = \Phi (a)$. \\
\textbf{I5.} $\neg (\Phi + \Psi) (a) = (\Phi + \Psi (\neg a))' = (\Phi (\neg a) + \Psi (\neg a))' = ('\Phi (\neg a))' \cdot (\Psi (\neg a))' = \neg\Phi (a)\cdot \neg \Psi (a) $. 
\end{proof}

\begin{proposition}\label{prop: isomorfismo G, G cap cap}
$\G\cong\hat{\hat{\G}}$.
\end{proposition}
\begin{proof}
We make good use of the duality established in \cite{Romanowska}, from which it follows $\A\cong\widehat{\widehat{\A}}$, where $\A$ is the GR space reduct of $\G$. To prove our claim we only have to prove that the isomorphism, given by \eqref{eq: isomorfismo A, Acapcap}, $\delta_{_A}(x)(\varphi) = \varphi (x)$, for $ x\in A $ and $\varphi\in\widehat{\A} $, preserve the involution. This is easily checked, indeed \\

$$(\neg\delta_{_A} (x))(\varphi) = (\delta_{_A} (x)(\neg\varphi))' = (\neg\varphi(x))' = (\varphi(\neg x))'' = \varphi(\neg x). $$
\end{proof}


\begin{lemma}\label{lemma: chiusura di S^ rispetto all'involuzione}
Let $\mathbf{I}\in\IBSL$ with $\mathbf{S}$ its bisemilattice reduct. If $\varphi\in\widehat{\S}$ then $\neg\varphi\in\widehat{\S}$. Moreover, $\widehat{\I}=\pair{\widehat{\S},\neg} $ is a GR space with involution. 
\end{lemma}
\begin{proof}
Suppose that $\varphi\in\widehat{\S}$, i.e. it is a map preserving sum and multiplication. It suffices to verify that also $\neg\varphi $ preserves the two operations. $\neg\varphi (x+y)= (\varphi(x+y)')' = (\varphi(x'\cdot y'))'= (\varphi (x')\cdot\varphi(y'))' = (\varphi(x'))' +(\varphi(y'))'= \neg\varphi(x) + \neg\varphi(y) $. For multiplication the proof runs analogously. 

For the second part, by \cite{Romanowska}, we have that $\widehat{\S}$ is a GR space, thus we only have to check that $\neg$ has the required properties. Let $\varphi,\psi\in\widehat{\S}$ and $x\in S$; properties $\mathbf{G1}-\mathbf{G4}$ can be easily verified as follows: \\

 $\neg (\neg \varphi (x)) = \neg (\varphi(x'))'=(\varphi(x''))'' = \varphi(x)$.\\

$\neg (\varphi\ast\psi)(x) = (\varphi\ast\psi (x'))' = (\varphi (x')\ast\varphi(x'))'= (\varphi(x'))'\ast(\psi(x'))' = \neg\varphi(x)\ast\neg\psi(x)$.\\

 Let $\varphi\leq\psi$, i.e. $\varphi(x)\leq_{\cdot}\psi(x)$ for each $x\in S$. In particular $\varphi(x')\leq_{\cdot}\psi(x')$, thus $(\psi(x'))' \leq_{+}(\varphi(x'))'$, i.e. $\neg \psi\sqsubseteq \neg \varphi $. \\

 Let $\varphi_0 $, $ \varphi_1$ and $\varphi_{\alpha}$ the constant homorphisms (of bisemilattices) on $0$, $1$ and $\alpha$, respectively.  $\neg\varphi_0 (x)=(\varphi_0(x'))' = 0' = 1 = \varphi_{1}(x)$; $\neg\varphi_1 (x)=(\varphi_1(x'))' = 1' = 0 = \varphi_{0}(x)$; $\neg\varphi_{\alpha} (x)=(\varphi_{\alpha}(x'))' = \alpha' = \alpha = \varphi_{\alpha}(x)$. \\

In order to prove \textbf{G5} and \textbf{G6}, it is enough to show that $\I\cong\widehat{\widehat{\I}}$. 
Recall that the bisemilattice reduct $\S$ of $\I$ is isomorphic to $\widehat{\widehat{\S}}$ under the isomorphism given by \eqref{eq: isomorfismo S, Scapcap}, namely $\varepsilon_{_S} (x)(\varphi) = \varphi (x)$, for every $\varphi\in\widehat{S}$ and $x\in S$. 
The map $\varepsilon_{_S} $ is obviously a homomorphism of bisemilattices and a bijection from $\I\setminus\{0,1\}$ to $\widehat{\widehat{\I}}\setminus\{\Phi_0, \Phi_{1}\}$, where by $\Phi_0, \Phi_1$ we indicate the constants in $\widehat{\widehat{\I}}$. This map can be extended to a bijection from $\I$ to $\widehat{\widehat{\I}}$, by setting $\varepsilon_{_S}(0) = \Phi_{0}$ and $\varepsilon_{_S}(1) = \Phi_{1}$. We have to prove that $\Phi_{0}$ and $\Phi_{1}$ indeed play the role of the constants in $\widehat{\widehat{\I}}$ and that $\varepsilon_{_S} $ also preserves involution. We start with the latter task: 
$$(\neg\varepsilon_{_S} (x))(\varphi) = (\varepsilon_{_S} (x)(\neg\varphi))' = (\neg\varphi(x))' = (\varphi(x'))'' = \varphi(x'). $$ 
Regarding the constants, we only need to prove that $\neg\Phi_{0} =\Phi_{1} $ and $\Psi + \Phi_{0} = \Psi$, for each $\Psi \in \widehat{\widehat{\I}}$. Indeed, for any $\varphi\in\widehat{\I}$, one has: 
$$\neg\Phi_{0} (\varphi) = \neg\varepsilon_{_S}(0)(\varphi) = \varphi (0') = \varphi (1) = \varepsilon_{_S}(1)(\varphi) = \Phi_{1} (\varphi). $$ 
Finally, due to the surjectivity of $\varepsilon_{S}$, for any $\Psi\in\widehat{\widehat{\I}}$, there exists $x\in I$ such that $\Psi = \varepsilon_{S}(x) $. Therefore
$\Psi (\varphi) = \varepsilon_{S}(x)(\varphi) = \varepsilon_{S}(x+0)(\varphi) = \varphi(x + 0) = \varphi(x) + \varphi(0) = \varepsilon_{S}(x)(\varphi) + \varepsilon_{S}(0)(\varphi) = (\Psi + \Phi_{0})(\varphi)  $ and we are done. 

\end{proof}

%

In order to prove Theorem \ref{th: stron-inv SA equivalente a GR inv} we are only left with proving that the functors $\mathrm{Hom_{b}(-,\WK)}: \mathfrak{IBSL}\to\mathfrak{IGR} $ and $\mathrm{Hom_{_{GR}}(-,\WK)}: \mathfrak{IGR}\to\mathfrak{IBSL} $ are controvariant (we consider just the first functor as for the other the proof runs analogously).
\begin{proposition}\label{prop: i funtori sono controvarianti}
Let $f:\I \to \mathbf{L}$ be a morphism of $\mathfrak{IBSL}$, then it induces a morphism of $\mathfrak{IGR}$ $f^{\ast}: \widehat{\mathbf{L}}\to\widehat{\I}$, where $\widehat{\mathbf{L}} $, $\hat{\I}$ are the dual spaces of $\mathbf{L}$ and $\I$, respectively.
\end{proposition}
\begin{proof}
$f^{\ast}$ is defined in the usual way, i.e. $f^{\ast}(\widehat{j})(i) = \widehat{j}(f(i))$, for each $i\in\I$ and $\widehat{j}\in\widehat{\mathbf{J}}$. It suffices to prove that $f^{\ast}$ preserves involution, namely $f^{\ast}(\neg\widehat{j}) = \neg f^{\ast}(\widehat{j})$, for all $j\in J$: 
 $$(\neg f^{\ast}(\widehat{j}))(i) = \neg \widehat{j}(f(i)) = f^{\ast}(\neg\widehat{j})(i),$$
\end{proof}

Surprisingly enough, we have established that semilattice inverse systems of Stone spaces are nothing but GR spaces with involution. 
\begin{theorem}\label{th: stron-inv SA equivalente a GR inv}
The categories of GR spaces with involution and $\cIBSL$ are dually equivalent. \end{theorem}
\begin{corollary}\label{corollario principale}
The category Sem-inv-$\SA$ is equivalent to the category of GR spaces with involution.
\end{corollary}

Corollary \ref{corollario principale} highlights an interesting as well as unexpected topological properties of Stone spaces. Indeed the  category of (semilattice) inverse systems of Stone spaces which deals with a possibly infinite family of them can be described by a specific class of topological spaces, namely GR spaces with involution.


\section{Final comments and remarks}\label{sec: commenti}

It is natural to wonder whether the content of Theorem \ref{th: stron-inv SA equivalente a GR inv} may be extended to other algebraic categories admitting topological duals such as bisemilattices and GR spaces.  Indeed, recall that bisemilattices are P\l onka sums of distributive lattices, according to the following
\begin{theorem}\cite[Th. 3]{Plo67a}\label{th: rappresentazione Plonka bisemilattices}
An algebra $\B$ is a bisemilattice iff it is the P\l onka sum over a semilattice direct system of distributive lattices.  
\end{theorem}

A \emph{Priestley space} is an ordered topological space, i.e. a set $X$ equipped with a partial order $\leq$ and a topology $\tau$, such that $\langle X, \tau \rangle$ is compact and, for $x\nleq y $ there exists a clopen up-set $U$ such that $x\in U$ and $y\not\in U$. The category of Priestley spaces, $\mathfrak{PS}$, is the category whose objects are Priestley spaces and morphisms are continuos maps preserving the ordering. 

The category of Priestley spaces is the dual of the category of distributive lattices \cite{Priestley72}, \cite{Priestley84}. 


Let us call $\mathfrak{BSL}$ the category of bisemilattices (objects are bisemilattices, morphisms homomorphisms of bisemilattices). It follows from our analysis and Theorem \ref{th: rappresentazione Plonka bisemilattices} that the objects in $\mathfrak{BSL}$ are the same as in Sem-dir-$\mathfrak{DL}$, where $\mathfrak{DL}$ stands for the category of distributive lattices.  We claim that the two categories of $\mathfrak{BSL}$ and Sem-dir-$\mathfrak{DL}$ are indeed equivalent. This can be shown using the same strategy applied in Section \ref{sec: dualita}.
\begin{lemma}\label{lemma: omomorfismi tra bisemireticoli}
Let $ \mathbf{L} $ and $ \mathbf{M} $ be two bisemilattices, the P\l onka sums over the semilattice direct systems of distributive lattices $\mathbb{L}= \langle L_i, \varphi_{i,i'}, I \rangle $ and  $ \mathbb{M}=\langle M_j, \varphi_{j,j'}, J \rangle $, and let $ h:\mathbf{L} \rightarrow \mathbf{M} $ be a homomorphism. Then, for any $ i\in I $, there exists a $j\in J $ such that $ h(L_i)\subseteq M_{j} $. \\
Moreover, there exists a semilattice homomorphism $ \varphi:I\rightarrow J $, for every homorphism $ h: \PL(\mathbb{L})\rightarrow\PL(\mathbb{M}) $, $ h(A_{i})\subseteq B_{\varphi(i)} $.
\end{lemma}
\begin{proof}
Let $ a,b\in L_i $: we claim that $h(a)$, $h(b)\in\M_j$, for some $j\in J$. Two cases may arise: either $a,b$ are comparable with respect to the order $\leq $ of $ L_i $ or they are not. Suppose $ a $ and $ b $ are comparable: let $ a\leq b $ and suppose that $ h(a)\in M_{j} $, $ h(b)\in M_{j'} $ with $ j\neq j' $. Then, $ h(a) = h(a\wedge b) = h(a)\wedge h(b)\in M_{j\vee j'} $ (by definition of operations in the P\l onka sum), therefore $ j= j\vee j'$. On the other hand, $ h(b) = h(a\vee b) = h(a)\vee h(b)\in  M_{j\vee j'} $. Thus $ j=j' $. 

The case of $ b <\ a $ can be proved anologously. 

Suppose now that $ a $ is not comparable with $ b $, namely $ a\not\leq b $ and $ b\not\leq a $. Clearly $ a\wedge b\leq a\vee b $, hence, reasoning as above, $ h(a\vee b) $ and $ h(a\wedge b) $ will belong to the same $ M_j $ for some $ j\in J $. Now, both $ a $ and $ b $ are comparable with $ a\wedge b $ and $ a\vee b $, hence necessarily $ h(a)\in M_j $ and $h(b)\in M_j $. Therefore $h(L_i)\in M_j$. 

The proof of the second statement runs analogously as for Lemma \ref{lem: omomorfismo degli indici}.
\end{proof}

\begin{remark}\label{rem: Lemma nel caso dei reticoli}
It is not difficult to check that the statement of Lemma \ref{lemma: morfismi tra somme di Plonka} can be proven analogously when considering semilattice direct systems of distributive lattices, instead of Boolean algebras, and morphisms between them.
\end{remark}

As consequence of Theorem \ref{th: rappresentazione Plonka bisemilattices}, Lemma \ref{lemma: omomorfismi tra bisemireticoli} and Remark \ref{rem: Lemma nel caso dei reticoli}, we get
\begin{theorem}\label{theo: BSL e equivalente a Sem-dir-DL}
The category $\mathfrak{BSL}$ is equivalent to \emph{Sem-dir-}$\mathfrak{DL}$.
\end{theorem}

Using Priestley duality and Theorem \ref{teorema: duality Sem-dir Sem-inv} we have
\begin{theorem}\label{th: dualita bsl}
The categories \emph{Sem-inv-}$\mathfrak{PS}$ and $ \mathfrak{BSL} $ are dually equivalent.
\end{theorem}
As the category of GR spaces is the dual category of $\mathfrak{BSL}$ (see Theorem \ref{th: dualita Romanowska}), this means that Sem-inv-$\mathfrak{PS}$ are equivalent to a single class of spaces, namely 
\begin{corollary}\label{cor: Sem-inv-PS sono GR spaces}
The category \emph{Sem-inv-}$\mathfrak{PS}$ is equivalent to the category of GR spaces.
\end{corollary}
\begin{center}
\textbf{Acknowledgments}
\end{center}
The first author acknowledges the support of the Horizon 2020 program of the European Commission: SYSMICS project, Proposal Number: 689176, MSCA-RISE-2015. The work of the first author is also supported by the Italian Ministry of Scientific Research (MIUR)  within the PRIN project ``Theory of Rationality: logical, epistemological and computational aspects.'' The second author was  supported by Prin 2015, Real and Complex Manifolds: Geometry, Topology and Harmonic Analysis (Italy). Finally, the authors thank Anna Romanowska and Jos\'e Gil-F\'erez for their valuable suggestions on the topics treated in the paper.


\end{document}